\documentclass[12pt]{amsart}
\usepackage{amssymb}
\usepackage{hyperref}
\newtheorem{Theorem}{Theorem}[section]
\newtheorem{Lemma}[Theorem]{Lemma}

\def \0{\lambda_{0}}

\newcommand{\R}{ \mathbb{R}}
\begin{document}
	\title[Solutions of Yamabe-Type  Equations on Projective Spaces  1]{Multiplicity of Solutions and Degenerate Solutions of Yamabe-Type  Equations on Projective Spaces  }

	\author{H\'ector Barrantes G.  }
	\address{Universidad de Costa Rica. Sede de Occidente. 20201. Alajuela. Costa
		Rica.}
	\email{ hector.barrantes@ucr.ac.cr}%
	
	\begin{abstract}We consider  Yamabe-type  equations on Projective Spaces
		$\mathbb{C} {\bf P}^n$  and $\mathbb{H} {\bf P}^n$ with the respectives canonical metrics,   and study the existence and multiplicity of solutions of Yamabe-type equation,  which are invariant by  cohomogeneity one actions of $ U(n)$ and $Sp(n)$ respectively.  We  also prove the existence of  degenerate solutions of the Yamabe-type  equationn on  $\mathbb{C} {\bf P}^n$  and $\mathbb{H} {\bf P}^n$,  which are invariant by  cohomogeneity one actions of $ U(n)$ and $Sp(n)$ respectively.
		
	\end{abstract}
	\maketitle

	\section{Introduction}
	
	Given a Riemannian manifold $(M^n , g)$ of dimension $n\geq 3$,  a
	{\it Yamabe-type } equation  is an elliptic equation of the form
	
	\begin{equation}\label{Yamabe-subcritical-1}
		-  \Delta_g u + \lambda u = \lambda u^{q-1}
	\end{equation}

	\noindent where  $\lambda >0$  and $2 < q < \frac{2n}{n-2}$. This equation has been studied by several authors in the particular case of  $\mathbb{S}^n$ with the round metric $g_0^n$. For instance  B. Gidas and J. Spruck proved in \cite{Gidas-Spruck}   that if $\lambda  \leq \frac{n-2}{4}$ the constant function $u = 1$ is the only solution.  In the article \cite{Bidaut-Veron} M.F. Bidaut and L. Veron showed  that  if $\lambda \leq \frac{n}{q-2}$  then the equation (\ref{Yamabe-subcritical-1}) has a unique solution (the constant solution).   They also showed that when  $\lambda$  is close $ \frac{n}{q-2}$, there are nonconstant solutions of (\ref{Yamabe-subcritical-1}).  In   \cite{YanYan}. Q. Jin, Y.Y. Li and H. Xu proved the existence and multiplicity of positive solutions, using bifurcation theory and the eigenvalues of the Laplacian on the sphere.  In \cite{Petean-Barrantes} H. Barrantes and J. Petean proved that, on $(\mathbb{S}^n \times \mathbb{S}^n ,   g_0^n + \delta g_0^n)$ with $\delta >0$, the equation (\ref{Yamabe-subcritical-1}) has positive solutions which depends non trivially on both factors.

	Motivation for studying Equation (\ref{Yamabe-subcritical-1}) also comes from  the study of the Yamabe equation .
	
	\begin{equation}\label{Eq-Yamabe}
		-a_n\Delta_g u + \mathbf{s}_gu  = \lambda u^{p-1},
	\end{equation}
	
	Where  $a_n =\frac{4(n-1)}{n-2}$, $p_n =\frac{2n}{n-2}$.  
	In case $u$ is a positive solution of this equation the
	conformal metric $u^{p_n -2} g$ has constant scalar curvature
	$\lambda$.  
	The Yamabe equation is the critical case of equation (\ref{Yamabe-subcritical-1}). Note that   when $q = p_n = \frac{2n}{n-2}$ and $\lambda =  \frac{\mathbf{s}_g}{a_n}$ the equation  (\ref{Eq-Yamabe}) is of the form of equation (\ref{Yamabe-subcritical-1} )
	
	We denote by $g_{FS}$ and  $g_{H}$ the canonical metrics on   $\mathbb{C} {\bf P}^n$  and $\mathbb{H} {\bf P}^n$.  Since $g_{FS}$ and  $g_{H}$  are   Einstein metrics, by   \cite{Obata}, in the critical cases 
	$ q= p_{2n} $ and $ q= p_{4n} $ respectively  and  $\lambda  = -\frac{S_{g_{FS}}}{a_{2n}}$, $\lambda  = -\frac{S_{g_{H}}}{a_{4n}}$,
	the only positive solution of the equation (\ref{Eq1-CPn})  is the constant solution $u\equiv 1$.

	In this article we  consider  the action   of the unitary group $U(n) $ on the complex projective space  $\left(\mathbb{C} {\bf P}^n, g_{FS}\right)$,  where  $g_{FS}$ is the   Fubini-Study metric. By restricting ourselves to  hypersurfaces and using local and global bifurcation theory techniques, we prove the existence and  multiplicity of positive solutions of the equation (\ref{Yamabe-subcritical-1}) on  $\mathbb{C} {\bf P}^n$.

	\begin{Theorem}\label{t1-CPn}
		For every positive integer $k$, let  $\lambda_k: = \lambda_{k, q} := \frac{4k(k+n)}{q-1} $ and   $q \in [2, p_{2n} )$.
		If    $\lambda \in \left(  \lambda_{k, q} , \;  \lambda_{k+1, q}\right]$,  then the equation
		
		\begin{equation}\label{Eq1-CPn}
			-\Delta_{g_{FS}} u + \lambda u  = \lambda u^{q-1}.
		\end{equation}
		
		\noindent  has at least $k$  positive solution  on     $(\mathbb{C} {\bf P}^n$,
		invariants by the action of  $U(n)$.
	\end{Theorem}
	
	Next, we consider  the action of the simplectic group $Sp(n)$ on the cuaternionic projective space 
	$ (\mathbb{H} {\bf P}^n$. This action are also of cohomogeneity one.  We will prove
	
	\begin{Theorem}\label{t2-HPn}
		For every positive integer $k$, let  $\mu_k := \mu_{k, q} : = \frac{4k(k+2n+1)}{q-1} $ and   $q \in [2, p_{4n} )$.
		If    $\lambda \in \left(  \mu_{k, q} , \;  \mu_{k+1, q}\right]$,  then the equation
		
		\begin{equation}\label{Eq1-HPn}
			-\Delta_{g_{H}} u + \lambda u  = \lambda u^{q-1}.
		\end{equation}
		\noindent  has at least $k$  positive solution  on     $(\mathbb{C} {\bf H}^n $
		invariants by the action of  $Sp(n)$.
	\end{Theorem}
	
	We  also prove the existence of degenerate solutions of equation (\ref{Yamabe-subcritical-1}) on 
	$\mathbb{C} {\bf P}^n$  and $\mathbb{H} {\bf P}^n$ which appear at the first bifurcation point, invariants by the action on  $U(n)$ and $Sp(n)$ respectively. 
	We call a solution $(u_0, \lambda_0)$ of Equation  (\ref{Yamabe-subcritical-1})   \textit{degenerate} if the linearized equation at $(u_0, \lambda_0)$ 
	
	$$
	-  \Delta_g v  + \lambda_0 v  = \lambda_0 (q-1)u_0^{q-2}v
	$$
	has a nontrivial solution. Studying the existence of degenerate solutions is  important to obtain information about the behavior of the branches of nonconstant solutions as well as the  exact number of solutions of Equation (\ref{Yamabe-subcritical-1}) using bifurcation techniques. We will prove:
	
	\begin{Theorem}\label{degen-CPn}
		There exists $\lambda > 0$ such that equation 
		$$-\Delta_{g_{FS}} u + \lambda u  = \lambda u^{q-1}.$$
		\noindent with  $q \in [2, p_{2n} )$,  has a nontrivial degenerate solution on   $(\mathbb{C} {\bf P}^n, g_{FS})$,
		invariant by the action  of $U(n)$.
	\end{Theorem}

	\begin{Theorem}\label{degen-HPn}
		There exists $\lambda > 0$ such that equation 
		$$-\Delta_{g_{H}} u + \lambda u  = \lambda u^{q-1}.$$
		\noindent with    $q \in [2, p_{4n} )$,  has a nontrivial degenerate solution on   $(\mathbb{H} {\bf P} ^n, g_{FS})$,
		invariant by the action  of the ymplectic group  $Sp(n)$.
	\end{Theorem}
	

	\section{Solutions of the ordinary differential equation on $\mathbb{C}{\bf P}^n$ near u =1 }
	
	\label{Section1}

	Given a solution $u  \colon \mathbb{C}{\bf P}^n  \to \R_{>0}$ of the equation (\ref{Eq1-CPn})
	we make $ w = u-1$.  Then, $u$ is a solution of (\ref{Eq1-CPn}) if and only if $w$ satisfies
	
	\begin{equation} \label{Eq1-w+1-CPn}
		-\Delta_{g_{FS}} w + \lambda (w+ 1) = \lambda (w+1)^{q-1}.
	\end{equation}
	
	\noindent We denote by $X_I$ the set of  functions on $\mathbb{C}{\bf P}^n$,   invariant by the action of  $U(n)$ and consider the Banach space 
	
	$$C^{2, \alpha} (X_I) := X_I \cap C^{2, \alpha}\left( \mathbb{C} {\bf P}^n \right)$$

	The linearization  of the equation (\ref{Eq1-w+1-CPn})  at $(0, \lambda)$  in   $v \in C^{0, \alpha} (X_I)$  is given by 
	\begin{equation}  \label{Eq-CPn-linear1}
		-\Delta_{g_{FS} }v =(q-2) \lambda v.
	\end{equation}
	
	\noindent It is known    \cite[P.173]{Berger-Gauduchon-Mazet}, that the eigenvalues of  $-\Delta_{g_{FS} }$ on $\mathbb{C}{\bf P}^n$
	are given by   $4k(k+n)$, with $ k =0, 1, 2,\ldots$.  Then the eigenvalues of problem  (\ref{Eq-CPn-linear1}) are given by
	$ \lambda_k :=\frac{4k(k+n)}{q-2}$ with $k=0, 1, 2, \ldots$.

	Now, consider the cohomogeneity one action  of $U(n)$  on  $\mathbb{C}{\bf P}^n$. The orbit space is isometric to
	$[0, \frac{\pi}{2}]$. Let $ f \colon \mathbb{C}{\bf P}^n \to [0,\frac{\pi}{2}]$ be the quotient  function. The singular orbits of this action are $ f^{-1}(0)$, where $p_0 \in \mathbb{C}{\bf P}^n$ is a fixed point under the action of  $U(n)$   and $f^{-1}\left(\frac{\pi}{2}\right)$. The regular orbits  are $M_{r} :=f^{-1} (r)$
	with  $ r \in (0, \frac{\pi}{2})$.  For more details see  \cite{Alexandrino-Betiol}.

	\noindent If   $w  \in C^{2, \alpha}(X_I)$ is a solution of (\ref{Eq1-w+1-CPn}), invariant by the action of $U(n)$,  then $w = \varphi \circ f $,
	where $\varphi \colon \left[0, \frac{\pi}{2}\right] \to \R$  and  $u$ is  $C^2$ if  $\varphi $ is  $C^2$. Moreover, the Laplacian of $w $ is given by 
	
	$$\Delta_{g_{FS}}w  =  (\varphi'' \circ f ) \bigl|\nabla_{g_{FS}} f \bigr|_{g_{FS}}^2  + (\varphi' \circ f )\Delta_{g_{FS}}f
	=(\varphi'' \circ f )   + (\varphi' \circ f )\Delta_{g_{FS}}f$$
	
	\noindent Therefore,  $w $ is a solution of   (\ref{Eq1-w+1-CPn}) if and only if $\varphi$ is a solution of 
	
	\begin{equation} \label{EDO-nonlinear-CPn}
		-\varphi''(r) - \left(\frac{2n \cos^2 (r)-1}{\cos (r) \sin (r)}\right)  \varphi'(r) + \lambda (\varphi(r) +1) =  \lambda (\varphi(r)+1)^{q-1}
	\end{equation}
	
	\noindent with the initial conditios and  $\varphi(0) =1 $ and $\varphi'(0) =0$ and such that $ \varphi'(\pi/2) =0 $.
	\noindent

	\noindent Similarly,   $v = \varphi \circ f \in C^{2, \alpha}(X_I)$  is a solution  of (\ref{Eq-CPn-linear1}), where  $\varphi$ es $C^2$, if an only if   $\varphi $  satisfies
	
	\begin{equation}\label{Eq-General-CPn-one-variable-linear}
		\varphi''(r) + \left(\frac{2n \cos^2 (r)-1}{\cos (r) \sin (r)}\right) \varphi'(r)  +\lambda (q-2) \varphi(r)=0
	\end{equation}
	
	\noindent where $r \in \left[0,  \frac{\pi}{2}\right]$ and $\varphi(0) =1 $, $\varphi'(0) =0= \varphi'(\pi/2) $.

	Now we shall prove  that for every positive integer
	$k$,  the equation  (\ref{Eq-General-CPn-one-variable-linear}) has a
	solution corresponding to the eigenvalue $\lambda_k$ and that this
	solution is a polynomial in  $\cos^2$.

	\begin{Lemma} \label{lemma-polynomial-CPn}
		For every positive integer  $k$, the solution of equation  (\ref{Eq-General-CPn-one-variable-linear}), with  $\lambda  = \lambda_k = \frac{4k(k+n)}{q-2} $
		has the form $$\varphi_k(r):= p_k(\cos^2 (r)),$$
		 \noindent  where $p_k$  is a polynomial o degree $k$ and     $r \in \left[0, \frac{\pi}{2}\right]$.
	\end{Lemma}
	
	\begin{proof}
		
		Define the operator 
		
		\begin{equation} \label{operator-L_mu}
			L_{\mu}(\varphi) :=  \varphi''(r) + \left(\frac{2n \cos^2 (r)-1}{\cos (r) \sin (r)}\right) \varphi'(r)  +\mu \varphi.
		\end{equation}
		
		\noindent Note that $L_{\mu} (\varphi)=0 $ is the same eigenvalue equation  (\ref{Eq-General-CPn-one-variable-linear}).  Denote by  $\varphi_{\mu}$ the corresponding solution of (\ref{Eq-General-CPn-one-variable-linear}) and   let $\mu = \alpha_k = 4k(k+n)$.   Then we can compute $\varphi_{\alpha_k}$ explicitely, for instance
		$$\varphi_{\alpha_1}(r)= \frac{n+1}{n} \cos^2(r) -\frac{1}{n}, $$
		$$\varphi_{\alpha_1}(r) = \frac{1}{n(n+1)}\Bigl((n+2)(n+3) \cos^4(r) -4(n+2)\cos^2(r) +2\Bigr)   $$

		\noindent Note that for every integer $m \geq1$ 
		
		$$L_{\mu}(\cos^{2m} t) = (\mu -4m(m+n))\cos^{2m} (r)+ 4m^2 \cos^{2m-2} (r) . $$

		So, for $k >2$ the formulas for $\varphi_k : =  \varphi_{\alpha_k}$ can then be found recursively
		
	\end{proof}
	
	Next we prove that every polynomial $\varphi_k$, found in
	the previous lemma, has $k$ single zeros in $\left(0,
	\frac{\pi}{2}\right)$.
	For this we use induction on $k$ and Sturm's comparison theorem.\\

	\begin{Lemma} \label{ceros-vk-CPn}
		The polynomial  $\varphi_k (r) = p_k(\cos^2(r))$ has  $k$ zeroes
		simples in $\left(0, \frac{\pi}{2}\right)$.
	\end{Lemma}
	
	\begin{proof}
		For $k = 1, 2$, the polynomials $\varphi_1$ and $\varphi_2$ have 1 and
		2 single zeros, respectively, in $\left(0, \frac{\pi}{2} \right) $. Suppose
		that it is also true for $k$ and let us prove that it holds for $k+1$.
		Let $r_1, r_2, \ldots, r_k$ be the $r_1, r_2, \ldots, r_k$ the
		zeros of $\varphi_k$ in $(0, \frac{\pi}{2})$.
		Since $\varphi_k$ and $\varphi_{k+1}$ are solutions of
		(\ref{Eq-General-CPn-one-variable-linear}) one has

		$$ \varphi_k''(r) + \left(\frac{ 2n \cos^2 (r)-1}{\cos (r) \sin (r)}\right) \varphi_k'(r)  +4k(n+k) \varphi_k(r) =0, $$
		$$ \varphi_{k+1}''(r) + \left(\frac{2n \cos^2 (r)-1}{\cos (r) \sin (r)}\right) \varphi_{k+1}'(r)  +4(k+1)(n+k+1) \varphi_{k+1}(r) =0$$
		
		\noindent Note that $4(k+1)( n+k+1)  >  4k(n+k),  $  and since 
		$$\varphi_k(0) =1, \;\;\;\;\varphi_k'(0) =0,\;\;\;\; \varphi_{k+1}(0) =1,\;\;\;\;  \varphi_{k+1}'(0)= 0,  $$
		
		\noindent  we have
		$$ \frac{ \varphi_{k+1}'(0)}{\varphi_{k+1}(0)} \leq \frac{\varphi_k'(0)}{\varphi_k(0) }$$
		
		\noindent Then, by  the Sturm's comparison theorem
		(see for instance \cite[Page 229]{Ince}),  between every two zeros of $\varphi_k$ there is at least one zero of $\varphi_{k+1}$ and the $i$-th zero of $\varphi_{k+1}$ is less than the $i$-th zero of $\varphi_k$, therefore $\varphi_{k+1}$ is less than the $i$-th zero of $\varphi_k$, so 	therefore $\varphi_{k+1}$ has at least $k$ single zeros: one in each interval $(0, r_1), \; (r_1 ,r_2), \ldots , (r_{k-1}, r_{k})$. 
		
		To see that $\varphi_{k+1} $ has a single zero in $(r_k, \frac{\pi}{2})$, note that $\varphi_{k+1}(r_k) \neq 0 $ since the $k$-zero of $\varphi_{k+1}  $  occurs before $r_k$. So, on the right side of $r_k$ we have.
		
		$$\frac{ \varphi_{k+1}'(r_k)}{ \varphi_{k+1}(r_k) } < \frac{ \varphi_{k}'(r_k)}{ \varphi_{k}(r_k)} = +\infty$$.
		
		\noindent Then, by Sturm's comparisons theorem, for every  $r \in \left(r_k, \frac{\pi}{2}\right)$.

		\begin{equation} \label{Inequality-Sturm-3}
			\frac{\varphi_{k+1}'(r)}{\varphi_{k+1}(r) } <  \frac{
				\varphi_{k}'(r)}{\varphi_{k}(r)}
		\end{equation}

		\noindent Suppose by contradiction that $\varphi_{k+1}$ does not vanishes in $\left(r_k, \frac{\pi}{2} \right) $ and also suppose that $\varphi_k'(r_k)<0$.
		Since $r_k$ is the last zero of $\varphi_k$ then $\varphi_k$ remains negative and 	decreasing in the interval 	$\left(r _k,\frac{\pi}{2} \right)$ and by the above assumption, $\varphi_{k+1}$ also remains negative in 
		$\left(r_k, \frac{\pi}{2}\right) $.
		
		\noindent Let us write	
		$$q(r) := \frac{2n \cos^2 (r)-1}{\cos (r) \sin (r)}. $$
		Note that 
		\begin{eqnarray*}
			\left(\frac{\varphi_k}{\varphi_{k+1}}\Bigl( \varphi_k'\varphi_{k+1 } - \varphi_k\varphi_{k+1}'\Bigr)\right)' &=& \left(\frac{\varphi_k'\varphi_{k+1} - \varphi_k\varphi_{k+1}' }{\varphi_{k+1}} \right)^2 + \frac{\varphi_k}{\varphi_{k+1}}q(r) \left(\varphi_k \varphi_{k+1}' -\varphi_k' \varphi_{k+1} \right) \\\\
			& & +\varphi_k^2 (\mu_{k+1}-\mu_k)
		\end{eqnarray*}
		
		\noindent The first and third terms on the right are both positive.  Let's look at the second term. By the inequality (\ref{Inequality-Sturm-3}) we have  $\varphi_k' \varphi_{k+1} -\varphi_k \varphi_{k+1}' > 0$ on  $\left(r_k, \frac{\pi}{2}\right) $.  
		
		\noindent On the other hand, note that
		\begin{equation}
			q(r)  < 0   \;\;\mbox{if} \;\; \cos (r) <
			\frac{1}{\sqrt{2n}}\;\;\;\mbox{and}\;\;\; q(r) >0 \;\;\mbox{if}\;\;
			\cos (r) > \frac{1}{\sqrt{2n}}.
		\end{equation}
		Let  $r_0 \in \left( 0, \frac{\pi}{2}\right)$ such that 
		$\cos (r_0) = 	\frac{1}{\sqrt{2n}}$.  Since the cosine function is decreasing at
		$\left( 0, \frac{\pi}{2}\right)$, we have 
		$q(r) < 0   \;\;\;\mbox{if}  \;\; \;  r > r_0    \;\;\;\;\mbox{and  } \;\; \;\;   q(r) > 0  \;\;\; \mbox{if} \; \;\; r < r_0$. 
		
		Suposse that $r_k < r_0$.  Integrating  from $r_0$ to
		$\frac{\pi}{2}$ and using that    $\varphi_{k+1}'\left(\frac{\pi}{2}\right)=0
		= \varphi_{k}'\left(\frac{\pi}{2}\right)$ we get 
		
	\begin{eqnarray*}
		-\left(\frac{\varphi_k}{\varphi_{k+1}}\left( \varphi_k'\varphi_{k+1}
		- \varphi_k\varphi_{k+1}' \right)\right)\left( r_0\right)
		&=& \int_{r_0 }^{\frac{\pi}{2}} \left(\frac{\varphi_k'\varphi_{k+1} - \varphi_k\varphi_{k+1}' }{\varphi_{k+1}} \right)^2dr  \\
		&&+ \int_{r_0}^{\frac{\pi}{2}} \frac{\varphi_k}{\varphi_{k+1}}q(r) \left(\varphi_k \varphi_{k+1}' -\varphi_k' \varphi_{k+1} \right)dr\\
		&& +  \int_{r_0}^{\frac{\pi}{2}} \varphi_k^2 (\mu_{k+1}-\mu_k)dr
	\end{eqnarray*}
	
	\noindent Note that the left-hand side of the equality is negative.
	On the other hand, by hypothesis both $\varphi_k$ and $\varphi_{k+1}$ remains 	negative in $\left(r_0, \frac{\pi}{2} \right) $, from which
	we obtain $ \frac{\varphi_k}{\varphi_{k+1}} >0$.
	It is also satisfied that 
	$\varphi_k \varphi_{k+1}' -\varphi_k' \varphi_{k+1} <0$ at
	$\left(r_0,  \frac{\pi}{2} \right)$ by the inequality (\ref{Inequality-Sturm-3}) and furthermore $q(r) < 0$ in 
	$\left(r_0, \frac{\pi}{2}\right)$. It follows that
	the right hand side of the above equality is positive, which is a
	contradiction.

	\noindent Now suppose that  $r_k >  r_0$. Analogous to the 
	previous case, we have that in $\left(r_k, \frac{\pi}{2}\right)$
	the following inequalities are satisfied 
	$ \frac{\varphi_k}{\varphi_{k+1}} >0$,    $\varphi_k \varphi_{k+1}' -\varphi_k' \varphi_{k+1} <0  $ y $q(r) > 0$.
	Integrating from $r_k$  to  $\frac{\pi}{2}$.
	and using the contitions  $\varphi_k(r_k) =0$ and 
	$\varphi_{k+1}'\left(\frac{\pi}{2}\right)=0 =
	\varphi_{k}'\left(\frac{\pi}{2}\right)$ we have
	
	\begin{eqnarray*}
		0 &= & \left.  \left(\frac{\varphi_k}{\varphi_{k+1}}\left( \varphi_k'\varphi_{k+1} - \varphi_k\varphi_{k+1}' \right)\right)'\right|_{r_k}^{\frac{\pi}{2}}\\
		&=&  \int_{r_k}^{\frac{\pi}{2}} \left(\frac{\varphi_k'\varphi_{k+1} - \varphi_k\varphi_{k+1}' }{\varphi_{k+1}} \right)^2dr  + \int_{r_k}^{\frac{\pi}{2}} \frac{\varphi_k}{\varphi_{k+1}}q(r) \left(\varphi_k \varphi_{k+1}' -\varphi_k' \varphi_{k+1} \right)dr\\
		&& +  \int_{r_k}^{\frac{\pi}{2}} \varphi_k^2 (\mu_{k+1}-\mu_k)dr\\
		& &  >  0
	\end{eqnarray*}
	
	\noindent which cannot be. Therefore $\varphi_{k+1}$ must
	have a zero to the right of $r_k$.  This finishes the proof of lemma (\ref{ceros-vk-CPn}).
\end{proof}


\section{Solutions of the ordinary differential equation on $\mathbb{H}{\bf P}^n$ near u =1 }
Now we study the number of zeros of  solutions of the  Yamabe-type equation  (\ref{Eq1-HPn}) on   $\mathbb{H}{\bf P}^n$.
In this case  $2 < q  < p_{4n}:= \frac{8n}{4n-2}$ and  $ g_{H}$ is the Fubini-Study metric.
Given a solution $u  \colon \mathbb{H}{\bf P}^n  \to \R_{>0}$ of the equation (\ref{Eq1-CPn})
we make $ \tilde{w} = u-1$.  Then, $u$ is a solution of (\ref{Eq1-CPn}) if and only if $w$ satisfies

\begin{equation} \label{Eq1-w+1-HPn}
	-\Delta_{g_{H}} \tilde{w} + \lambda (\tilde{w}+ 1) = \lambda (\tilde{w}+1)^{q-1}.
\end{equation}

The linearization  of the equation (\ref{Eq1-w+1-HPn})  at $(0, \lambda)$  in   $v \in C^{0, \alpha} (X_I)$  is given by 
\begin{equation}  \label{Eq-HPn-linear1}
	-\Delta_{g_{H} }v =(q-2) \lambda v.
\end{equation}

\noindent It is known    \cite[P.202]{Besse}, that the eigenvalues of  $-\Delta_{g_{H} }$ on $\mathbb{H}{\bf P}^n$
are given by   $4k(k+2n+ 1)$, with $ k =0, 1, 2,\ldots$.  Then the eigenvalues of problem  (\ref{Eq-HPn-linear}) are given by
$ \mu_k :=\frac{4k(k+2n+1)}{q-2}$ with $k=0, 1, 2, \ldots$

Consider the cohomogeneity one action  of $Sp(n)$  on  $\mathbb{H}{\bf P}^n$. The orbit space is isometric to $[0, \frac{\pi}{2}]$. 
Let $ f \colon \mathbb{H}{\bf P}^n \to [0,\frac{\pi}{2}]$ be the quotient  function. The singular orbits of this action are $p_0  = f^{-1}(0)$, where $p_0 \in \mathbb{H}{\bf P}^n$ is a fixed point under th action of  $Sp(n)$   and $\mathbb{H} {\bf P}^{n-1} = f^{-1}\left(\frac{\pi}{2}\right)$. The regular orbits  of this action are  $M_{r} :=f^{-1} (r)$ wit  $ r \in (0, \frac{\pi}{2})$.  For more details see  \cite{Alexandrino-Betiol}.
\noindent 	We denote by $Y_I$ the set of smooth functions on $\mathbb{H} {\bf P}^n$ invariant under the action of the simplectic group $Sp(n)$ on $\mathbb{H} P^n$. Let us consider the Banach space

$$C^{2, \alpha} (Y_I) := Y_I \cap C^{2, \alpha}\left( \mathbb{H}{\bf P}^n \right). $$ 
If   $u  \in C^{2, \alpha}(Y_I)$ is a solution of (\ref{Eq1-HPn}), then $u = \psi \circ f $,
where $\psi \colon \left[0, \frac{\pi}{2}\right] \to \R$  and  $u$ is  $C^2$ if  $\psi $ is  $C^2$ and  $\psi'(0) =0 = \psi'(\pi/2) $.
Therefore,  $u $ is a solution of   (\ref{Eq1-HPn}) if and only if $\psi$ is a solution of 

\begin{equation} \label{EDO-nonlinear-HPn}
	-\psi''(r) - \left(\frac{(4n +2)\cos^2 (r)-3}{\cos (r) \sin (r)}\right)  \psi'(r) + \lambda \psi(r) =  \lambda \psi(r)^{q-1}
\end{equation}

\noindent Similarly,   $v = \psi \circ f \in C^{2, \alpha}(Y_I)$  is a solution  of (\ref{Eq-HPn-linear}), where  $\psi$ es $C^2$, if an only if   $\psi $  satisfies

\begin{equation}\label{Eq-General-HPn-one-variable-linear}
	\psi''(r) + \left(\frac{(4n +2)\cos^2 (r)-3}{\cos (r) \sin (r)}\right) \psi'(r)  +\lambda (q-2) \psi(r)=0
\end{equation}

\noindent where $r \in \left[0,  \frac{\pi}{2}\right]$. We are looking for solutions such that  $\psi'(0) =0 = \psi'(\pi/2) $.

Now we    show that for every positive integer
$k$,  the equation  (\ref{Eq-General-HPn-one-variable-linear}) has a
solution corresponding to the eigenvalue $\mu_k$ and that this
solution is a polynomial in  $\cos^2$. Also we prove   that this solution contains $k $ zeros in the interval $\left(0, \frac{\pi}{2}\right)$ all of them  simples.

\begin{Lemma} \label{lemma-polynomial-HPn}
	For every positive integer  $k$, the solution of equation  (\ref{Eq-General-HPn-one-variable-linear}), with  $\lambda  = \mu_k = \frac{4k(k+2n+1)}{q-2} $
	has the form $$\psi_k(r):= \tilde{p}_k(\cos^2 (r)),$$   where $\tilde{p}_k$  is a polynomial o degree $k$ and     $r \in \left[0, \frac{\pi}{2}\right]$.
\end{Lemma}

\begin{proof}
	
	Define the operator 
	
	\begin{equation} \label{operator-L_mu}
		\tilde{L}_{\mu}(\psi) :=  \psi''(r) + \left(\frac{(4n+2) \cos^2 (r)-3}{\cos (r) \sin (r)}\right) \psi'(r)  +\mu \psi.
	\end{equation}
	
	\noindent Note that $\tilde{L}_{\mu} (\psi)=0 $ is the same eigenvalue equation  (\ref{Eq-General-HPn-one-variable-linear}).  Denote by  $\psi_{\mu}$ the corresponding solution of (\ref{Eq-General-HPn-one-variable-linear}).  Let $ \mu = \alpha_k = 4k(k+2n+1)$  then we can compute $\psi_{\alpha_k}$ explicitely.  For instance
	
	\begin{equation}\label{polyn-HPn}
		\psi_{\alpha_1}(r)= \frac{n+1}{n} \cos^2(r) -\frac{1}{n},  
	\end{equation}
	$$\psi_{\alpha_2}(r) =   \frac{1}{n(2n+1)}  \left( (n+2)(2n+3) \cos^4 (r)-3(2n+3) \cos^2 (r)+ 3 \right)$$

	\noindent Note that for every integer $m \geq1$ 
	
	$$\tilde{L}_{\mu}(\cos^{2m} t) = (\mu -4m(m+2n+1))\cos^{2m} (r)+ 4m (m+1) \cos^{2m-2} (r) . $$

	So, for $k >2$ the formulas for $\psi_k := \psi_{\alpha_k}$ can then be found recursively
	
\end{proof}

Next we prove that every polynomial $\tilde{p}_k$, found in the previous lemma, has $k$ single zeros in $\left(0,
\frac{\pi}{2}\right)$.
For this we use induction on $k$ and Sturm's comparison theorem.\\

\begin{Lemma} \label{ceros-vk-HPn}
	The polynomial  $\psi_k (r) = \tilde{p}_k(\cos^2(r))$ has  $k$ zeroes
	simples in $\left(0, \frac{\pi}{2}\right)$.
\end{Lemma}

\begin{proof}
	The proof of this lemma is completely analogous to the proof of lemma (\ref{ceros-vk-CPn})
\end{proof}

\section{ Proof of Theorem 1.1 }
In this section we proof  Theorem \ref{t1-CPn}.  The proof  is analogous to that of Theorem (2.2) in \cite{Petean-Barrantes}. We write here for the reader convenience.    We use the lemmas from the   section \ref{Section1},
to prove that every eigenvalue $\lambda_{k}$ is a bifurcation point. Then we use the local and global bifurcation theory to study the set of non
nontrivial solutions  that appears at each bifurcation point.

\begin{proof}[Proof of Theorem \ref{t1-CPn}. ]
	
	We define the operator 
	$$  S:C^{2, \alpha} \left( X_I \right) \times \R_{\geq 0} \rightarrow C^{0, \alpha} \left( X_I \right),   $$
	$$ S(w,\lambda )= -\Delta_{g_{FS}} w +\lambda (w+1 -(w+1)^{q-1} ).$$

	For any $\lambda \in \R_{\geq 0}$, we have  $S(0,\lambda )=0$.
	Let us study nontrivial solutions $(w, \lambda)$ of $S(w,\lambda )=0$ that branch from the curve $(0, \lambda)$. 
	Differentiating  $S$ with respect to $w$ at $(0, \lambda)$ we have.
	
	$$S_u ' (0,\lambda) [v] = -\Delta_{g_{FS}} v - \lambda (q-2) v .$$
	\noindent Then $S_u ' (0,\lambda) [v]=0$ if and only if.
	
	\begin{equation} \label{equation-eigenvalues-CPn}
		-\Delta_{g_{FS}} v - \lambda (q-2) v=0.
	\end{equation}

	\noindent For each positive integer $k$, and $\displaystyle\lambda_k:
	=\frac{4k(n+k)}{q-2}$, we denote  $L_k = S_u ' (0,\lambda_k) $. By  lemma (\ref{lemma-polynomial-CPn}),
	we know that if $ \lambda = \lambda_k$, then the corresponding solution of 
	(\ref{Eq-General-CPn-one-variable-linear}) has the  form   $\varphi_{k} (r) = p_k (\cos^2 (r))$, where $p_k$ is a
	polynomial of degree $k$. Therefore $ \ker (L_k) = \langle \varphi_k\rangle$.  That is, $\ker (L_k)$ has dimension 1. \noindent Note
	that by integration by parts, if $w \in C^{2, \alpha}
	\left( \mathbb{C} {\bf P}^n \right)$ we have.
	
	$$ 0 = \int_{\mathbb{C} {\bf P}^n} L_k(\varphi_{k}) \ w dv_{g_{FS}} = \int_{\mathbb{C} {\bf P}^n}  L_k(w) \varphi_{k} \ dv_{g_{FS}} .$$
	\noindent This implies that the rank $R(L_k)$ of $L_k$, is
	
	$$R(L_k) = \left\{ y \in C^{0, \alpha} \left(\mathbb{C} {\bf P}^n \right) :    \int_{ \mathbb{C} {\bf P}^n } y\varphi_k \ dv_{g_{FS}} =0 \right\} $$
	\noindent On the other hand, differentiating $S'_u$ with respect to $\lambda$ in $(0,\lambda_k)$ and evaluating at $\varphi_k$ we have
	$ S_{u, \lambda} ''(0,\lambda_k)[\varphi_{k}] = (q-2)\varphi_{k}, $  and  since $ \int_{\mathbb{C} {\bf P}^n}    \varphi_k^2 \; dv_{g_{FS}}  \neq 0 $ we obtain that
	$S_{u, \lambda} '' (0,\lambda_k)[\varphi_{k}] \notin R(L_k). $ Therefore, by (\cite[Theorem 2.8]{Ambrosetti-Malchiodi}),
	the bifurcation points of $S(0, \lambda ) =0$ are the points $ \lambda_k $. Moreover,  by the Crandall-Rabinowitz Theorem 
	(\cite{Crandall-Rabinowitz}) near to $(0, \lambda_k)$   the branch of non-trivial solutions can be parametrized by
	
	\begin{equation} \label{parametrization-Crandall-Rab-CPn}
		s  \mapsto(w(s) , \lambda (s) ) \; \mbox{with}\;\; w(s) = s\varphi_k + \beta_k(s)\; 
		, \beta_k (0) =0= \beta'_{k}(0) , \, \, \mbox{and} \;  \lambda (0) = \lambda_k
	\end{equation}
	
	Let $C$ be the closure of the set of nontrivial, positive solutions of $S(w,\lambda )=0$ in $C^{2,\alpha} (X_I)$ and let $C_k$ be the connected component of $C$ containing the point  $(0,\lambda_k )$. By the Rabinowitz global bifurcation theorem  ( see \cite[Theorem 4.8]{Ambrosetti-Malchiodi} or \cite[theorem 3.4.1]{Nirenberg}) we have that either $C_k$ is noncompact or $C_k$ contains a point $(0,\lambda_i )$ with $i \neq k$. We are going to prove that $C_k $ is noncompact.
	
	According to Lemma  \ref{ceros-vk-CPn}  $\varphi_k$ has $k$ zeros simples in $(0,\pi/2 )$, thus  
	the corresponding solution $\varphi$  of (\ref{EDO-nonlinear-CPn}),  has $k$ zeros simples  in $(0,\pi/2 )$ for $s$ small.  Then, if $(w(s) , \lambda (s)$  is a nontrivial solution of (\ref{Eq1-w+1-CPn}) for any $r$ in $(0,\pi/2 )$  in $C_k$,  there is an open around $w$, where all the corresponding solutions of the ordinary differential equation (\ref{EDO-nonlinear-CPn}) has the same number of zeros as $\varphi$ for $s $ small. This implies that  $(0,\lambda_i )$ does not belong to $C_k$ if $i \neq k$.

	\noindent Let $K \colon C^{2, \alpha}\left(X_I \right) \to C^{2,
		\alpha}\left(X_I \right)$  
	the inverse of the operator
	
	$$-\Delta_{g_{FS}} +Id \colon C^{4, \alpha}\left(X_I \right) \to C^{2, \alpha}\left(X_I \right)$$
	
	\noindent The operator $K$ is linear and compact. Let us consider the
	region
	
	$$D:= \{(w, \eta) \in C^{2, \alpha}  \left(X_I \right)\times \R :w >-1, \eta > 1 \} , $$
	
	\noindent We define $T \colon D \to C^{2, \alpha}\left(X_I \right)$
	by
	
	$$T(w, \eta) = \frac{\eta -1}{q-2} K\left( (w+1)^{q-1} -(q-1)w-1) \right ).$$
	
	\noindent Note that $T$ is a compact operator and for every
	$ \eta >1$ we have
	$T(0, \eta) =0, T_w'(0, \eta) =0 $. Now we define  $ F\colon D \to C^{2, \alpha}\left(X_I \right) $ by.
	
	$$F(w, \eta) = w- \eta K(w) -T(w, \eta)$$.
	
	\noindent Note that $F(0, \eta) =0 $ for each $\eta$. And if we apply $-\Delta_{g_{FS}}  +Id$ to the equation $F(w, \eta) =0$, we see that  $F(w, \eta) =0$ if and only if.
	
	$$-\Delta_{g_{FS}}  w -\frac{\eta -1}{q-2} \left((w+1)^{q-1} - (w+1)\right)=0.$$
	
	\noindent Therefore, $F(w, \eta) =0$ if and only if $w$ is a solution of the equation (\ref{Eq1-w+1-CPn}) for $\lambda = \frac{\eta -1}{q-2}$.
	
	Let  $\eta_k = \lambda_k (q-2) +1$. Similarly, as before, let $B$ be the closure of the nontrivial solutions $(w,\eta)$ of
	$F(w,\eta)=0$ in $D$ and let $B_k$ be the connected component of $B$ containing the point $(0,\eta_k )$. Then, by the
	Rabinowitz's global bifurcation theorem \cite[Theorem 4.8]{Ambrosetti-Malchiodi}, it follows that either $B_k$ is noncompact
	or $B_k$ contains another point $(0,\eta_j)$ with $j\neq k$, with $\eta_j$ bifurcation point of $F(w, \eta )$=0.
	But we have already seen that the second condition is not satisfied, so
	therefore $B_k$ is non-compact. But 
	$C_k = \left\{ \left(w, \frac{\eta- 1}{q-2} \right) : (w,\eta ) \in B_k \right\}$
	and therefore $C_k$ is noncompact.
	
	\vspace{.3cm}
	
	Now, since    $ g_{FS}$ has positive Ricci curvature and by \cite[Theorem 6.1]{Bidaut-Veron},  there exist  $\rho>0$
	such that  if $\lambda < \rho$ the equation  (\ref{Eq1-CPn}) only has the trivial solution. For any $\lambda_0$, $0< \rho < \lambda_0$, the set 
	$$A:=\{ (w,\lambda ): S(w, \lambda )=0 , \  \lambda \in [\rho , \lambda_0 ] \}, $$
	We are going to prove that $A$ is compact.
	The proof of this  is well-known, see for instance the proof of \cite[Lemma 2.2]{YanYan}. Note that there exists $\Lambda >0$ such that if $(w,\lambda ) \in A$ then
	$w \leq \Lambda$. This is proved by the blow up technique (see for instance the proof in \cite[Theorem 2.1, page 200]{Schoen-Yau}): let  $(w_i  , \lambda_i ) $ be a sequence in $A$, with $w_i = u_i-1$  and let 	$x_i \in \mathbb{C} {\bf P}^n$ such that
	$w_i (x_i ) \rightarrow \infty$ i.e $u_i (x_i ) \rightarrow \infty$. Taking a subsequence we can assume that $x_i \rightarrow x \in \mathbb{C} {\bf P}^n$ 	and $\lambda_i \rightarrow \lambda \in [\rho , \lambda_0 ]$. Then by taking a normal neighborhood of $x$ and renormalizing $u_i$ one would construct as a limit a positive solution of $\Delta u +\lambda u^{p-1} =0$
	in $\R^{2n}$. But since
	$q<p_{2n}$ is subcritical such solution does not exist by \cite{Gidas-Spruck}.  Then we consider  the
	compact operator  $K \colon C^{2, \alpha}\left(X_I \right) \to C^{2, \alpha}\left(X_I\right)$ defined above which is the inverse of $-\Delta_{g_{FS}} +Id$    and point out that $S(w,\lambda )=0$ if and only if
	$u=K(\lambda u^{p-1} -(\lambda -1) u)$: this implies that $A$ is compact.

	If there exists $\lambda_* > \lambda_k $ such that it does not exist $w \neq 0$ such that  
	$(w, \lambda_* ) \in C_k$, then,  since $C_k$ is connected we have that $C_k \subset C^{2,\alpha} (X_I ) \times [\rho ,\lambda_* ]$. But this  would imply that $C_k$ is compact, which is a contradiction. Then for any $\lambda > \lambda_k$ there exist $w \neq 0$ such that
	$(w, \lambda ) \in C_k$. Since $C_k \cap C_j =\emptyset$ if $j \neq k$, this proves Theorem (\ref{t1-CPn}).

\end{proof}

\section{ Proof of Theorem 1.2 }
In this section we proof  Theorem \ref{t2-HPn}.   In an analogous way to the case of $\mathbb{C}{\bf P}^n$  We use the lemmas from the previous section, to prove that every eigenvalue $\lambda_{k}$ is a bifurcation point. Then we use the local and global bifurcation theory to study the set of non
nontrivial solutions  that appears at each bifurcation point.

\begin{proof}[Proof of Theorem \ref{t1-CPn}. ]
	
	We use bifurcation theory. 	 Define the operator
	
	$$  F:C^{2, \alpha} \left( Y_I \right) \times \R_{\geq 0} \rightarrow C^{0, \alpha} \left( Y_I \right).$$
	$$ F(\tilde{w},\lambda )= -\Delta_{g_{H}} \tilde{w} +\lambda (\tilde{w} + 1 -(\tilde{w} +1)^{q-1} ).$$

	For any $\lambda \in \R_{\geq 0}$, we have  $F(1,\lambda )=0$.
	Let us study nontrivial solutions $(\tilde{w}, \lambda)$ of $F(\tilde{w},\lambda )=0$ that branch from the curve $(0, \lambda)$. Differentiating  $F$ with respect to $\tilde{w}$ at $(0, \lambda)$ we have.
	
	$$F_{\tilde{w}} ' (0, \lambda)[v]  -\Delta_{g_{H}} v - \lambda (q-2) v .$$
	\noindent Then $F_{\tilde{w}}' (1,\lambda) [v]=0$ if and only if.

	\begin{equation}  \label{Eq-HPn-linear}
		-\Delta_{g_{H} }v =(q-2) \lambda v.
	\end{equation}

	\noindent For each positive integer $k$, and $ \lambda = \displaystyle \mu_k:
	=\frac{4k(k+2n +1)}{q-2}$, we denote  $\tilde{L}_k = F_{\tilde{w}} ' (1,\lambda_k) $ and using lemmas (\ref{lemma-polynomial-HPn}) and ( \ref{ceros-vk-HPn}), the rest of the proof is completely analogous to that of the theorem
	(\ref{t1-CPn}).

\end{proof}

\section{Proof of Theorems 1.3 and 1.4}

As in the previous section, we denote  by $C$ the closure of the set of nontrivial solutions of $S(w, \lambda)=0$.
According to the lemma  (\ref{lemma-polynomial-CPn}),  the eigenspace associated to  $\lambda_k$
has dimension 1  and is generated by a polynomial $\varphi_k$ of degree  $k$ in  $\cos^2(r)$.  By the Crandall-Rabinowitz Theorem  (\cite{Crandall-Rabinowitz}), near to $(0, \lambda_k)$,   the connected  component of  $C$ containing the point  $(0, \lambda_k)$ can be parameterized by
$( w(s), \lambda(s) )$,  with  
\begin{equation}\label{psi}
w(s) = s \varphi_k + \beta_k (s), \;\;\;	\lambda(0) = \lambda_k, \;\;\;\beta_k (0)= 0 = \beta_k' (0)
\end{equation}


\begin{Lemma}\label{lema-lambda'(0)-CPn}
Let $( w(s), \lambda (s)  )$  be the  curve of nontrivial solutions of  $S(w, \lambda)=0$,
	which appear at the bifurcation point $(0 ,\lambda_1)$. then $\lambda'(0) \neq 0$. $\lambda'(0) \neq 0$.
\end{Lemma}

\begin{proof}

	Let $k \geq 1$.Replacing (\ref{psi})  in equation (\ref{Eq1-w+1-CPn}) we have
	$$
	\Delta_{g_{FS}}\bigl( s \varphi_k +\beta_k (s)\bigr) - \lambda(s)
	\bigl( s \varphi_k + \beta_k (s) +1 - \bigl( s \varphi_k + \beta_k (s) +1
	\bigr)^{q-1}\bigr) =0.
	$$
	\noindent \noindent Differentiating with respect to $ s$ we obtain
	
	\begin{eqnarray*}
		\Delta_{g_{FS}} \bigl(\varphi_k+ \beta_k' (s)\bigr) - \lambda'(s) \Bigl[ s \varphi_k + \beta_k (s) +1 - \left( s \varphi_k + \beta_k (s) +1 \right)^{q-1}\Bigr] & &  \\
		- \lambda(s) \Bigl[ \varphi_k + \beta_k' (s)  - (q-1)\bigl( s \varphi_k
		+ \beta_k (s) +1 \bigr)^{q-2}\left(\varphi_k + \beta_k'(s)\right)
		\Bigr] \; = \; 0
	\end{eqnarray*}

	\noindent And differentiating again with respect to $s$ and  evaluating  at $s=0$ we obtain

	\begin{equation} \label{lambdaprima1}
		\Delta_{g_{FS} }\beta_k'' (0)  + (q-2)\lambda_k  \beta_k'' (0)
		+(q-1)(q-2)\lambda_k  \varphi_k^2 +2(q-2)\lambda'(0)\varphi_k  =0
	\end{equation}
	
	\noindent Multilplying by $\varphi_k$ and integrating by parts   on  $ \mathbb{C}{\bf P}^n $ we obtain

	$$ \lambda'(0)= \frac{-(q-1) \lambda_k \displaystyle \int_{\mathbb{C}{\bf P}^n} \varphi_k^3  dv_{g_{FS}} }{2(q-2)\displaystyle \int_{\mathbb{C}{\bf P}^n} \varphi_k^2  dv_{g_{FS}}}  $$
	
	\noindent  In order to prove the lemma, it is enough to analize the integral in the numerator with  $k=1$.
	Note that  $\varphi_1$ is given by 
	$\varphi_1(r) =\frac{n+1}{n}\cos^2 (r) - \frac{1}{n}$ then

	\begin{eqnarray*}
		\int_{\mathbb{C}{\bf P}^n} \varphi_1^3  dv_{g_{FS}} &=&\int_{S^{2n-1}} \int_0^{\frac{\pi}{2}} \left(\frac{n+1}{n}\cos^2 (r) - \frac{1}{n}\right)^3  \sin^{2n -1}(r) \cos (r) dr dv_{g_0^{2n-1}} \\
		&=&Vol\left(S^{2n-1}, {g_0^{2n-1}}\right)\frac{1}{n^3} \int_0^{\frac{\pi}{2}} \left((n+1)\cos^2 (r) - 1\right)^3  \sin^{2n -1}(r) \cos (r) dr \\
		\end{eqnarray*}
		
		\noindent Making the substitution $t = \sin (r)$ we have
		
		\begin{eqnarray*}
			\int_{\mathbb{C}{\bf P}^n} \varphi_1^3  dv_{g_{FS}}  
			&=& \frac{1}{n^3}Vol\left(S^{2n-1}, {g_0^{2n-1}}\right) \frac{1}{2(n+2)(n+3)} \left( -7n-2\right)  \neq 0
		\end{eqnarray*}
		
		In a similar way for the case of  $ \mathbb{H}{\bf P}^n $ we have that

		$$ \lambda'(0)= \frac{-(q-1) \mu_k \displaystyle \int_{\mathbb{H}{\bf P}^n} \psi_k^3  dv_{g_{H}} }{2(q-2)\displaystyle \int_{\mathbb{H}{\bf P}^n} \psi_k^2  dv_{g_{FS}}}  $$
		
		and for $k = 1$, acording to \ref{polyn-HPn} we have that

		\begin{eqnarray*}
			\int_{\mathbb{H}{\bf P}^n} \psi_1^3  dv_{g_{H}} &=&\int_{S^{4n-1}} \int_0^{\frac{\pi}{2}} \left(\frac{n+1}{n}\cos^2 (r) - \frac{1}{n}\right)^3  \sin^{4n -1}(r) \cos^3 (r) dr dv_{g_0^{4n-1}} \\
		\end{eqnarray*}

		Making again  the substitution $t = \sin (r)$ we have $\lambda'(0) \neq 0$
	\end{proof}
	
	
	Now we prove the Theorem  \ref{degen-CPn}.
	
	\subsection{ Proof of  Theorem \ref{degen-CPn}.}
	\begin{proof}

		We will only prove Theorem 1.3 since the proof of 1.4 is completely analogous.  Since the metric  $g_{FS}$ has positive Ricci curvature,  by \cite[Theorem 6.1]{Bidaut-Veron}, we know that there exists value $\eta >0 $, such that if $\lambda < \eta$ then the equation $S(w, \lambda)=0 $ has only the trivial solution.
		\noindent Consider the set 
		
		$$D:=\{ (w,\lambda ): S(w, \lambda )=0 , \  \lambda \in [\eta , \lambda_1 ] \}, $$

		We know that near  $(0, \lambda_1)$, 	the nontrivial solutions  have the form 	$(w(s), \lambda(s))$, with $w(s ) = s \varphi_1 + \beta_1(s)$ and $\lambda(0) =\lambda_1$, $\beta_1(0) =0= \beta_1'(0)$ and by the previous lemma  $\lambda'(0) \neq 0$.
		
		This implies that there exist values $ \lambda \in [\eta, \lambda_1)$, for which the equation $S(w, \lambda)=0$ has nontrivial solutions. Then the function $ \lambda$ has a minimum $ \lambda_* \in [\eta, \lambda_1)$. Note that  the connected  component $C_1$ of  nontrivial solutions containning   $(0, \lambda_1)$ satisfies $C_1 \cap D \subset D$, which shows that $D$ is nonempty.  Note also that $D$ is compact.
		Let $(\lambda_j)$ be a sequence in $[\eta, \lambda_1)$ that converges to $\lambda_*$ with a solution $w_j$ of $S(w, \lambda) =0$ with $\lambda= \lambda_j$.
		That is, the sequence $(w_j, \lambda_j)$ satisfies $S(w_j, \lambda_j) =0$.
		Since $D$ is compact then the sequence $(w_j, \lambda_j)$ converges to a point $(w_*, \lambda_*) \in D$.
		Therefore, $(w_*, \lambda_*)$ is a solution of the equation $S(w, \lambda) =0$.
		
		Since $ \lambda_* < \lambda_1$ and $\lambda_1$ is the first bifurcation point of $S(w, \lambda) =0$, then $w_*$ is a nontrivial solution.
		We now prove that $(w_*, \lambda_*)$ is a degenerate solution of $S(w, \lambda)=0$.
		Suppose by contradiction that $(w_*, \lambda_*)$ is nondegenerate. The derivative of $S$, with respect to $w$ at $(w_*,\lambda_*)$, is given by.
		
		$$S_w' \colon  C^{2,  \alpha}( X_I) \times \R \to  C^{2, \alpha}( X_I) $$
		$$S_w'(w_*, \lambda_* )[v, \lambda] = \Delta_{g_{FS}}v -\lambda_* (1-(q-1)w_*^{q-2})v$$.
		
		If $(w_*, \lambda_*)$ were non-degenerate, the equation.
		
		$$ \Delta_{g_{FS}}v -\lambda_* (1-(q-1)w_*^{q-2})v=0$$
		
		\noindent would have only the trivial solution $v \equiv 0 $.  This means that the operator $S$ is a local diffeomorphism in the direction of $w$ at the point $w_*$.  Then, there are values $\lambda < \lambda_*$ for which equation $S(w, \lambda)=0$ has nontrivial solutions, which contradicts the fact that $\lambda_*$ is the minimum value of $\lambda$, for which there are nontrivial solutions.
		Therefore, the solution $(w_*, \lambda_*)$ is a degenerate solution of $S(w, \lambda) =0$. Then, taking $u_* := w_* +1$,
		we have that $(u_*, \lambda_*)$ is a degenerate solution of the equation (\ref{Eq1-CPn}).
		This completes the proof of the Theorem \ref{degen-CPn} .
	\end{proof}
	

\end{document}